\documentclass[12pt,reqno,twoside]{amsart}
\usepackage{amsmath}
\usepackage{amssymb}

\newtheorem{theorem}{Theorem}[section]
\newtheorem{lemma}[theorem]{Lemma}
\newtheorem{proposition}[theorem]{Proposition}

\theoremstyle{definition}

\theoremstyle{remark}

\numberwithin{equation}{section}

\newcommand\vare{\varepsilon}

\newcommand\s{\sigma}

\newcommand\RR{\mathbb R}
\newcommand\ZZ{\mathbb Z}

\newcommand\cD{\mathcal{D}}
\newcommand\cE{\mathcal{E}}
\newcommand\cH{\mathcal{H}}
\newcommand\cM{\mathcal{M}}
\newcommand\cS{\mathcal{S}}

\newcommand\cP{\mathcal{P}}

\newcommand\norm[1]{\left\|#1\right\|}

\newcommand\abs[1]{\left|#1\right|}

\newcommand\inn[1]{\left\langle #1 \right\rangle}
\newcommand\set[1]{\left\{{#1}\right\}}

\begin{document}

\title[Equidistribution in $SL_2(\RR)$-actions]{Equidistribution in measure-preserving actions \\ of semisimple groups : case of $SL_2(\RR)$}

\author{Amos Nevo}
\address{Department of Mathematics, Technion IIT}
\email{anevo@tx.technion.ac.il}
\thanks{The author was supported by ISF Grant}




\dedicatory{}


\begin{abstract}
We prove  pointwise convergence for the semi-radial averages on $G=SL_2(\RR)$ given by 
$\int_t^{t+1} m_K\ast \delta_{a_{s}}ds$ (and similar variants),  acting on $K$-finite $L^p$-functions in a probability-measure-preserving action of the group, for $p > 1$.  
\end{abstract}

\maketitle

\section{Pointwise convergence for semi-radial averages on $SL_2(\RR)$}


Let $a_t = \begin{pmatrix} e^t & 0 \\ 0 & e^{-t} \end{pmatrix}$, and
let $K=SO(2,\RR) \subset SL(2,\RR)$. Let $(X,\mu)$ be a standard Borel  
measure space on which $SL(2,\RR)$ acts, preserving 
the probability measure $\mu$ 
on $X$. We let $\pi_X$ denote the unitary representation of $G$ in $L^2(X)$, given by $\pi_X(g)f(x)=f(g^{-1}x)$. Recall that $f\in L^2(X)$ is called a $K$-finite function if the linear span $V(f)$ of the set of vectors $\set{\pi_X(k)f\,;\, k\in K}$ is finite dimensional. 

Assume $\eta \in C_c(\RR)$ is a 
non-negative bump function of unit integral, and define the averaging operators 
$$\cM^\eta_t f(x)= \int_{-\infty}^\infty \eta( t -s) \left(
  \int_K f( a_s k x) \, dm_K(k) \right) \, ds $$
  where $m_K$ is normalized Haar probability measure on $K$.

\begin{theorem}\label{thm:sl(2)}
\label{sl2}
Assume $f \in
L^{1+\kappa}(X,\mu)$ for some $\kappa > 0$, and $f$ is a $K$-finite function. If $\mu$ is an ergodic measure, then
for
$\mu$-almost every  $x$ in $X$, 
$$
\lim_{t \to \infty} \cM^\eta_t f(x)= \int_X f(y) \, d\mu(y).
$$

\end{theorem}

We remark that the same pointwise convergence result holds for any (not necessarily ergodic) invariant measure $\mu$, with the limit being the conditional expectation on the sub-$\sigma$-algebra of $G$-invariant sets. Furthermore, the same conclusion holds when $t\to -\infty$ as well. 

We note that a mean ergodic theorem for semi-radial averages $m_K\ast \delta_{a_t}$ is proved in greater generality in \cite{Ve1}. The pointwise convergence result stated in Theorem  \ref{thm:sl(2)} is utilized in \cite{EM1}, which provided the motivation for its formulation. Another application of semi-radial averages can be found in \cite{EMM}.

\section{Ergodic theorems for character-spherical averages on $SL_2(\RR)$}

\subsection{The maximal inequality for character-spherical operators}
Before proceeding to the proof of 
Theorem~\ref{thm:sl(2)} we begin with some preliminaries. First, let 
$\psi_n(k_{\theta})=\exp( in\theta)$ denote the characters of the 
circle group $SO(2,\RR)=K=\{k_\theta\mid 0\le \theta \le 2\pi\}$. Let $\chi_n$ denote 
the (complex) Radon measure on $K$ whose density w.r.t. normalized Haar measure on $K$ is 
$\overline{\psi}_n$. To any bounded Borel measure $\nu$ on $G$, there corresponds 
in any strongly continuous unitary representation $\pi$,
the operator $\pi(\nu)v=\int_G \pi(g)vd\nu(g)$. This correspondence is a bounded homomorphism of the convolution algebra, namely $\pi(\nu_1 \ast \nu_2)=\pi(\nu_1)\circ \pi(\nu_2)$, and $\norm{\pi(\nu)}\le \norm{\nu}$ (where $\norm{\nu}$ is the total variation norm).

In particular the operator
$\pi(\chi_n)$ is equal, in 
any strongly continuous unitary representation $(\pi, \cH^{\pi})$ of 
$K$, to the orthogonal projection onto the 
closed linear subspace $\cH_n^{\pi}=\{v\in \cH^{\pi}\mid \pi(k_\theta)v
=\psi_n(k_\theta)v\}$. In particular, $\pi(\chi_n)$ is a self-adjoint 
idempotent operator.

Now consider the (complex) measure $\sigma_t(n,m)=\chi_n\ast \delta_{a_t}
\ast \chi_m$, where $\delta_{a_t}$ denotes the delta measure at the element 
$a_t\in A$. The corresponding operator $\pi(\sigma_t(n,m))$
 has norm bounded by 
$1$, but is not self-adjoint, in general. If $\eta\in C_c(\RR)$, 
satisfies $\eta\ge 0$ and 
$\int_{-\infty}^{\infty}\eta(s)ds=1$, the measure
$\gamma^\eta_t(n,m)=\int_{-\infty}^\infty \eta (t-s)\sigma_s (n,m)ds$, 
 is a convex combination of the measures $\sigma_s(n,m)$, where 
$s$ ranges over $t-B_\eta$ with $B_\eta=\text{supp} \,\,\eta$.
In general, the measures of the form $\chi_n\ast \nu\ast \chi_m$ are precisely the complex measures $\omega$ on $G$ satisfying the equation $\chi_n\ast \omega\ast \chi_m =\omega$, and will be called character-spherical measures. 

Note that if $f\in L^2(X)$ is $K$-finite, then there is a finite set $F_f\subset \ZZ$ of characters of the circle group $K$, such that the finite-dimensional space $V(f)$ spanned by the $K$-translates of $f$ is a finite direct sum $\bigoplus_{n\in F_f} V_n(f)$, where $V_n(f)$ is finite-dimensional, and $\pi(k_\theta)$ acts on $V_n(f)$ as scalar multiplication by $\psi_n(k_\theta)=e^{in \theta}$. In particular, $\pi(\chi_n)f$ is the orthogonal projection of $f$ on $V_n(f)$  for every $n\in F_f$, and $\sum_{n\in F_f}\pi(\chi_n)f=f$. Conversely, every function $f\in L^2(X)$ satisfying that $\sum_{n\in F}\pi(\chi_n)f=f$ for some finite set $F$ of characters of $K$,
 is a $K$-finite function.

The pointwise convergence stated in Theorem \ref{thm:sl(2)} is  based on the following maximal inequality. 

\begin{theorem}
\label{theorem:maximal}
In any probability-measure-preserving action $(X,\mu)$
 of $SL_2(\RR)$, the operators 
$\gamma^\eta_t(n,m)$ (for any $n,m$) satisfy the strong type $(p,p)$ global maximal inequality in every 
$L^p$, $1 < p \le \infty$, namely :
$$\|\sup_{t \in \RR}
\left|\pi(\gamma^\eta_t(n,m))f\right|\|_p\le C_p(\eta)\|f\|_p$$
The constant $C_p(\eta)$ is independent of $f$, and also of $m$ and $n$.
\end{theorem}

\begin{proof} Clearly, for almost every $x\in X$ :
$$\abs{ \pi(\chi_n\ast\delta_{a_s}\ast\chi_m)f(x)}
=\abs{\int_K \int_K f(k_{\theta_1}^{-1}a_s^{-1}k_{\theta_2}^{-1} )\overline{\psi}_m(k_{\theta_1})\overline{\psi}_n(k_{\theta_2})dm_K(k_{\theta_1})dm_K(k_{ \theta_2})}$$
$$\le 
\pi(m_K\ast\delta_{a_s}\ast m_K) |f(x)|$$
Denoting $\sigma_t=\sigma_t(0,0)=m_K\ast \delta_{a_t}\ast m_K$ and $\gamma^\eta_t=\gamma^\eta_t(0,0)=\int_\RR \eta(t-s) \sigma_s ds $, we therefore have 
$$\left|\pi(\sigma_t(n,m))f(x)\right|\le 
 \pi(\sigma_t)|f(x)| \text{   and  }
\abs{\pi(\gamma^\eta_t(n,m))f(x)}\le \pi(\gamma^\eta_t)\abs{f(x)}.$$

The maximal inequality for $\gamma^\eta_t$ in every $L^p(X)$, $1 < p \le \infty$
is a straightforward consequence of the maximal inequality for the operators 
$\gamma_t=\int_t^{t+1} \sigma_s ds$ that was established in \cite[\S3, Prop.3]{NS}. Indeed, in the reference cited the operator considered was $\sup_{t \ge 1} \gamma_t$, but the intervals $[t,t+1]$ can be replaced by intervals of any fixed length. Furthermore the local operator $\sup_{0\le t \le 1} \gamma_t$ certainly satisfies the maximal inequality, by the local transfer principle. Recalling that $\sigma_t=m_K\ast \delta_{a_t}\ast m_K=m_K\ast \delta_{a_{-t}}\ast m_K=\sigma_{-t}$,  we can consider arbitrary intervals in $\RR$. Finally, the maximal inequalities for averages which are uniform on intervals clearly imply the corresponding results averages defined by non-negative compactly supported bump functions. 
\end{proof}

\subsection{Pointwise convergence of character-spherical operators}
We now turn to the proof of the following pointwise convergence result, that will be utilized in the proof of Theorem \ref{thm:sl(2)} below.

\begin{theorem}
\label{theorem:pointwise}
In any finite-measure-preserving action $(X,\mu)$
 of $SL_2(\RR)$, and for any $m\in \ZZ$, the sequence 
$\pi(\gamma_t^\eta(0,m))f(x)$ converges, as $t \to \infty$ and as $t\to -\infty$, for every $f\in L^p(X)$, $1 < p <
\infty$, pointwise almost every (and in the $L^p$-norm). The limit is zero, 
unless $m=0$, in which case the limit is $\cE f (x)$, where $\cE$ is the
projection of $f$ on the space of $G$-invariant functions. 

\end{theorem}

As is well known, given the strong maximal inequality stated in Theorem \ref{theorem:maximal},  in order to prove Theorem \ref{theorem:pointwise} it suffices to establish 
the existence of a dense subspace
of functions $f\in L^2(X)$ for which 
$\pi(\gamma_t(0,m))f(x)$ converges almost everywhere to the stated limit. Indeed, if 
$\norm{f-f_k}_p\to 0$, where $f_k$ belong to the dense subspace then 
$$\abs{\pi(\gamma_t(0,m))f(x)-\pi(\gamma_s(0,m))f(x)}\le $$
$$\abs{\pi(\gamma_t(0,m))(f-f_k)(x)}+
\abs{(\pi(\gamma_t(0,m))-\pi(\gamma_s(0,m)))f_k(x)}+
\abs{\pi(\gamma_s(0,m))(f_k-f)(x)}$$
so that 
$$\limsup_{t,s\to \infty}\abs{\pi(\gamma_t(0,m))f(x)-\pi(\gamma_s(0,m))f(x)}\le 2\sup_{t > 0}\abs{\pi(\gamma_t(0,m))(f-f_k)(x)}+0$$
and hence 
$$\norm{\limsup_{t,s\to \infty}\abs{\pi(\gamma_t(0,m))f(x)-\pi(\gamma_s(0,m))f(x)}}_p\le 2C_p \norm{f-f_k}_p\longrightarrow 0\,.$$
To establish the existence of such a subspace
we  will consider the spectral decomposition of the unitary representation
$\pi_X$, and use the decay estimates of $K$-finite functions and their first derivatives in irreducible unitary representations of $G$. 

We begin with some preliminaries.
First, let $(\tau,\cH^\tau)$ be an irreducible non-trivial strongly continuous 
unitary representation of $G=SL_2(\RR)$. Then 
$\cH^\tau=\sum_{n\in {\ZZ}}\cH_n^\tau$, where $\cH_n^\tau$ is the 
closed subspace that affords the representation with character $\psi_{n}$ of the 
circle group $K$. As is well known (see e.g. \cite[Ch.3]{HT}), 
the dimension of $\cH_n^\tau$ is either zero or one, 
depending on $\tau$. 
Let $v_n^\tau$ denote a unit vector 
in $\cH_n^\tau$ if such a vector exists, and the zero vector otherwise. 
Given $v\in \cH^\tau$, we write $v=
\sum_{k\in {\ZZ}}\inn{v,v_k^\tau} v_k^\tau$, so that $\norm{v}^2=\sum_{k\in \ZZ} \abs{\inn{v,v_k^\tau}}^2$ .
 Recall that 
$\tau(\chi_m)=\tau(\chi_m)^\ast$ is the orthogonal projection operator on $\cH_m^\tau$, so that 
$\tau(\chi_m)v=\inn{v,v_m^\tau}v_m^\tau$, and $\abs{\inn{v,v_m^\tau}}^2=\norm{\tau(\chi_m)v}^2$. Since $\sigma_t(n,m)=\chi_n\ast\delta_{a_t}\ast\chi_m$,  Bessel identity implies
$$\norm{\tau(\sigma_t(n,m))v}^2=
\sum_{k\in {\ZZ}}|\inn{\tau(\sigma_t(n,m)) v,v_k^\tau}|^2
=\sum_{k\in {\ZZ}}|\inn{\tau(a_t)\tau(\chi_m) v ,
\tau(\chi_n)v_k^\tau}|^2$$
$$ = \abs{\inn{v,v_m^\tau}}^2 \abs{\inn{\tau(a_t)v_m^\tau, v_n^\tau}}^2
=\abs{\inn{v,v_m^\tau}}^2\abs{\Phi_{m,n}^\tau(a_t)}^2\,,$$
where the matrix coefficient $\Phi_{m,n}^\tau (g)=\inn{\tau(g)v_m^\tau,v_n^\tau}$
is called the $(m,n)$-spherical function associated with the representation $\tau$.

The same argument shows that 
$$\norm{\tau(\sigma_t(n,m))v-\tau(\sigma_s(n,m))v}^2=
\abs{\inn{v,v_m^\tau}\cdot \inn{ \tau(a_t)v_m^\tau,v_n^\tau}
-\inn{v,v_m^\tau}\cdot \inn{\tau(a_s)v_m^\tau,v_n^\tau}}^2=
$$
$$= \abs{\inn{v,v_m^\tau}}^2\abs{\Phi_{m,n}^\tau(a_t)-\Phi_{m,n}^\tau(a_s)}^2
$$

Taking $s=t+h$, dividing by $h$
 and letting $h$ tend to zero, 
we conclude that if $v\in \cH^\tau$ is a 
$C^\infty$-vector of the representation $\tau$, then 
$$\norm{\frac{d}{dt}\tau(\sigma_t(n,m))v}^2=
\abs{\inn{v,v_m^\tau}\frac{d}{dt}\Phi_{m,n}^\tau(a_t)}^2$$
which is an explicit form for the multiplier operator corresponding 
to the
differentiation operator given by 
$\frac{d}{dt}(
\chi_n\ast\delta_{a_t}\ast\chi_m)$.

\subsection{Estimates of $K$-finite functions}
We turn to stating the estimates of generalized spherical functions and their derivatives that we will use below. We denote the unitary spectrum of $SL_2(\RR)$ by $\Sigma$, and consider the character-spherical functions $\Phi^\tau_{m,n}(a_t)=\inn{\tau(a_t)v_m^\tau, v_n^\tau}$, for $\tau\in \Sigma\setminus\set{1}$ (here $1$ denotes the trivial representation). 
We will give a simple direct argument for the spectral estimates we will actually use, following the exposition in \cite[Ch. V, \S 3.1]{HT}. First consider the function space of smooth even functions on $\RR^2\setminus\set{0}$, homogeneous of degree $z$ :
$$\cS^{z,+}=\set{f\in C^\infty(\RR^2\setminus\set{0})\,;\, f(ts)=\abs{t}^z f(s)}$$
and the corresponding space of odd functions :
$$\cS^{z,-}=\set{f\in C^\infty(\RR^2\setminus\set{0})\,;\, f(ts)=\text{ sgn} (t)\abs{t}^z f(s)}$$
$SL_2(\RR)$ acts of these function spaces via the usual (linear) action on $\RR^2\setminus\set{0})$, and we denote this action by $\rho(g)f=f\circ g^{-1}$. The (even) principal series of unitary representations is realized on the spaces 
$\cS^{z,+}$ corresponding to $z=-1+i\RR$, 
and the (odd) principal series on the spaces $\cS^{z,-}$ with $z=-1+i(\RR\setminus\set{0})$. 
The complementary series representation is realized on $\cS^{z,+}$ with $z\in (-2,0)$. Write $z=-1-\overline{\alpha}$, and define a bilinear form on $\cS^{-1-\overline{\alpha},\pm}\times \cS^{-1+\alpha,\pm}$ by integrating the function values on the unit circle  
$$\inn{f,h}=\int_0^{2\pi} f(\cos \theta,\sin \theta)\overline{h(\cos \theta,\sin \theta)}d\theta\,.$$
This bilinear form is invariant under the action $\rho$ of $SL_2(\RR)$ on the two function spaces involved, namely for $f\in \cS^{-1-\overline{\alpha},\pm}$ and $h\in \cS^{-1+\alpha,\pm}$
$$\inn{\rho (g)f,\rho(g)h}=\inn{f,h}\,.$$
 The associated matrix coefficients are given along $A$ by  
$$\inn{f,\rho(a_t)h}=\int_0^{2\pi}f(\cos \theta,\sin \theta)\overline{h(e^{-t}\cos \theta, e^{t}\sin \theta)} d\theta$$
$$=\int_0^{2\pi}f(\cos \theta,\sin \theta)(e^{-2t}\cos^2\theta+e^{2t}\sin^2\theta)^{\frac12(-1+\overline{\alpha})}\overline{h(\cos \beta, \sin \beta)} d\theta$$
where $\beta$ is a function of $t$ and $\theta$. 
If follows immediately that if $\abs{f}=\abs{h}=1$ on the unit circle, then 
\begin{equation}\label{zonal}
\abs{\inn{f,\rho(a_t)h}}\le \int_0^{2\pi}(e^{-2t}\cos^2\theta+e^{2t}\sin^2\theta)^{\frac12(-1+\text{Re } \overline{\alpha})} d\theta
\end{equation}

The standard estimate of the positive-definite spherical functions on $SL_2(\RR)$ will be recalled presently. But beforehand, note that we will also require in our analysis estimates for the {\it derivative} of the character-spherical functions, and we will develop here the estimates that we will actually use below by a simple direct argument. 

Let us choose now $h$ to be the $K$-invariant vector in the representation space (when it exists), whose restriction to the unit circle is the constant $1$, and also choose $f$ as the vector affording the representation $\psi_n$ (when it exists), so that its restriction to the unit circle is $e^{in\theta}$. The resulting expression is 
$$\inn{e^{in\theta},\rho(a_t)1}=\int_0^{2\pi}e^{in\theta}(e^{-2t}\cos^2\theta+e^{2t}\sin^2\theta)^{\frac12(-1+\overline{\alpha})} d\theta$$
When the parameter $z$ is such that the representation on $\cS^{z,\pm}$ (and its dual) is a unitary representation $\tau$ (namely even or odd principal series, or complementary series), the preceding expression is precisely $\overline{\inn{\rho(a_t)1,e^{in\theta}}}=\overline{\Phi^\tau_{0,n}}(a_t)$. 
Note also that 
 $$\inn{1,\rho(a_t)(e^{in\theta})}=\overline{\Phi^\tau_{n,0}}(a_t)=\inn{\rho(a_{-t})1,(e^{in\theta})}
=\Phi^\tau_{0,n}(a_{-t})$$
Thus $\Phi^\tau_{n,0}(a_t)=\overline{\Phi^\tau_{0,n}}(a_{-t})$.

The parametrization of the irreducible non-trivial
 unitary representations
of $SL_2(\RR)$ that we will use below is given in \cite[Ch. III, \S 1.3,Thm. 1.3.1]{HT}, as follows.

The representation $\rho$ on the space $\cS^{(-1+i\lambda,\pm)}$, $\lambda\in \RR$ will be denoted by $\pi_{\lambda}^{\pm}$, and constitutes the even and odd principal series, where $\lambda\in \RR$ in the first case, and $\lambda\in \RR\setminus \set{0}$ in the second case. The representation $\rho$ on  $\cS^{-1-s}$, $0 < s < 1$ will be denoted by $\pi_s$, and constitutes the complementary series. We note that $\pi_s$ is also the representation realized on $\cS^{-1+s}$, $0 < s < 1$. 
Finally, we denoted by  $\pi^{(k,\pm)}$, $k \ge 1$ the representations of 
 the discrete series ($k \ge 2$) and limits of discrete series ($k=1$).

We let $\Xi$ denote the Harish Chandra 
$\Xi$-function, given by  $\Xi= \Phi_{0,0}^{\pi_0^+}(a_t) $.
Recall that $\Xi$-function satisfies the estimate :
$0 < \Xi(t) \le C(1+t)\exp(-t)$.

Before stating the spectral estimates which will be relevant to our discussion, let us note (see \cite[Ch. III, Prop. 1.2.6]{HT})  that the odd principal series representations $\pi_\lambda^-$ do not contain a $K$-invariant unit vector, so that the matrix coefficients $\Phi_{n,0}^{\pi_\lambda^-}$ and $\Phi_{0,n}^{\pi_\lambda^-}$ defined above are equal to zero. It is well-known (see e.g. 
\cite[Ch. III]{HT}) that the discrete series representations and the limit of discrete series representations also do not contain a $K$-invariant unit vectors, so that the same conclusion applies. 

We now note the following fact.

\begin{lemma}
\label{estimates}
There exists a constant $B$ such that the following holds. 

Let $z=-1+i\lambda$ ($\lambda\in \RR$)  be a principal series parameter, and $\pi_\lambda^+$ be the associated irreducible unitary 
representation of $SL_2(\RR)$. Then 
\begin{itemize}
\item[{\rm (1)}]
$\abs{\Phi_{0,n}^{\pi_\lambda^+}(a_t)}\le  \Xi(a_t)\le B(1+t)\exp(-t)$.
\item[{\rm (2)}]
 $\abs{\frac{d}{dt}\Phi_{0,n}^{\pi_\lambda^+}(a_t)}
\le  B(1+\abs{\lambda}) (1+t)\exp(-t)$.

Let $z=-1-s$ ($s\in (0,1)$)   be a complementary series parameter, and let $\pi_s$ be the associated 
irreducible unitary  representation. Then

\item[{\rm (3)}]
$\abs{\Phi_{0,n}^{\pi_s}
(a_t)}\le B\cdot\exp (-(1-s)t)$
\item[{\rm (4)}]
$\abs{\frac{d}{dt}
\Phi_{0,n}^{\pi_s}
(a_t)}\le B
\cdot\exp (-(1-s)t)$
\item[{\rm (5)}]
The same estimates apply to the functions $\Phi_{n,0}^\pi$ with $\pi$ a principal series or complementary series representation. 
\item[{\rm (6)}] All other irreducible non-trivial representations $\tau$ do not contain a $K$-invariant unit vector and the matrix coefficients $\Phi_{n,0}^\tau$ and $\Phi_{0,n}^\tau$ are identically zero. 
\end{itemize}
\end{lemma}

\begin{proof} 
Consider first the representations of the even principal series $\pi_\lambda^+$ 
or the complementary series $\pi_s$, which we denote temporarily by $\tau_z$. 
As noted above, to cover all cases it suffices to estimate the derivative of 
$$\Phi_{0,n}^{\tau_{-1-\overline{\alpha}}}(a_t)= \inn{ \rho (a_t)1, e^{in\theta}}=\frac{1}{2\pi}
\int_0^{2\pi}e^{-in\theta}
\left(e^{-2t}\cos^2 \theta+e^{2t}\sin^2 \theta\right)^{\frac12(-1+\text{Re } \overline{\alpha})}d\theta$$

Differentiating, we obtain : 
 $$\abs{\frac{d}{dt}\Phi_{0,n}^{\tau_{-1-\overline{\alpha}}}(a_t)}=$$
 $$= \abs{\frac12(-1+\text{Re }\overline{\alpha})\cdot 2\int_0^{2\pi}e^{-in\theta}  
\left(e^{-2t}\cos^2 \theta+e^{2t}\sin^2 \theta\right)^{\frac12(-1+\text{Re }\overline{\alpha})}\cdot 
\frac{\sin^2 \theta e^{2t}-\cos^2 \theta e^{-2t}}{\sin^2 \theta e^{2t}+\cos^2 \theta e^{-2t}}   d\theta}
$$
$$ \le (1+\abs{\alpha}) \int_0^{2\pi}\left(e^{-2t}\cos^2 \theta+e^{2t}\sin^2 \theta\right)^{\frac12(-1+\text{Re }\overline{\alpha})}d\theta
=(1+\abs{\alpha})\Phi_{0,0}^{\tau_{-1+\text{Re }\overline{\alpha}}}(a_t)\,, $$
using
$$\abs{ \frac{\sin^2 \theta e^{2t}-\cos^2 \theta e^{-2t}}{\sin^2 \theta e^{2t}+\cos^2 \theta e^{-2t}} }\le 1\,.$$
Thus the proof of parts (1)-(4) is complete upon recalling the estimate of Prop. 3.1.5 
of \cite[Ch. V,\S 3.1]{HT} for the standard spherical function, which asserts the desired exponential decay estimates.

\end{proof}

We conclude :

\begin{proposition}\label{final-estimate}
For every irreducible non-trivial unitary representation $\tau$ of $SL_2(\RR)$ with parameter $(-1-\overline{\alpha},\pm)$,  the following estimates hold, for all $n\in \ZZ$, with $B$ an absolute constant.  

\begin{equation}\label{mc-estimte}
\abs{\Phi^\tau_{0,n}(a_t)}\le B(1+\abs{t})\exp(-\vare_\tau \abs{t})\,,
\end{equation}

\begin{equation}\label{derivative-mc-estimate} \abs{\frac{d}{dt}\Phi^\tau_{0,n}(a_t)}\le 
B(1+\abs{\alpha})(1+\abs{t})\exp(-\vare_\tau \abs{t})\,,
\end{equation}
with $\vare_\tau=1$ for $\tau=\pi_\lambda^{\pm}$ ($\lambda\in \RR$) and 
$\vare_\tau=1-s$ for $\tau=\pi_s$, ($0 < s < 1$). 
We denote $C(\tau)=B(1+\abs{\alpha})$. 
\end{proposition}

%

\subsection{ Proof of Theorem~\ref{theorem:pointwise}}
Given a (cyclic) unitary representation $\pi$ of $SL_2(\RR)$, we consider 
a direct integral decomposition of $\pi$, writing $\pi =\int_{\tau\in\Sigma}^\oplus \tau d\zeta_\pi(\tau)$, where 
$\Sigma$ denotes the unitary spectrum of the group and $\zeta_\pi$ the spectral
 measure. Given $n$ and $\vare > 0$, we define the following subset of the spectrum $\Sigma$, using the parameters $\vare_\tau$ and $C(\tau)$ defined in Proposition 
 \ref{final-estimate}.  
$$\Sigma_\vare=
\set{\tau\in \Sigma\,;\, \vare_\tau > \vare, C(\tau) < \frac{1}{\vare}}\,.$$

We let $\cP_{\epsilon}$ denote the orthogonal 
projection on the subspace $\cH_\epsilon^\pi$ of vectors $v\in \cH^\pi$ 
whose spectral measure is supported in the set $\Sigma_\epsilon$. 
$\cP_\epsilon$ commutes with all the operators $\pi(g)$, $g \in G$, and hence also with $\pi(\nu)$ for any complex measure $\nu$ on $G$. In particular it commutes with all the operators $\pi(\sigma_t(n,m))$, $\pi(\gamma_t(n,m))$ as well as their derivatives.

 Let $v\in C^\infty( \cH_\vare^\pi)$, namely $v$ is a $C^\infty$ vector belonging to the subspace of vectors whose spectral support is in $\cH_\vare$. 
 
 Let 
 $v=\int_{\tau\in \Sigma_\pi}^\oplus v^\tau d\zeta_{\pi,v}(\tau)$ denote the direct integral decomposition of $v$. 
 We use the 
expression found in \S 2.2 for the spectral multiplier of 
the derivative operator together with the foregoing
spectral estimates, and conclude: 
$$\norm{\frac{d}{dt}\pi(\sigma_t(0,n))v}^2=\int_{\Sigma_\epsilon}
|\inn{v^\tau,v_n^\tau}|^2
\abs{\frac{d}{dt}\Phi_{n,0}^\tau(a_t)}^2d\zeta_{\pi,v}(\tau)$$
$$
\le \frac{1}{\vare^2} (1+\abs{t})^2 \exp (-2\abs{t}\vare) \cdot \norm{v}^2\,.$$

In the next subsection, using the argument of  \cite[\S 7.1]{N1}, the pointwise convergence of $\pi(\sigma_t(0,n))f(x)$ to zero 
for almost every $x\in X$ will be established below for every 
$f\in \cup_{\epsilon > 0}\cD(\cH_\epsilon^\pi)$, where $\cD(\cH_\vare^\pi)$ is a subspace of $C^\infty(\cH_\vare^\pi)$, norm-dense in $\cH_\vare^\pi$. Given this fact, pointwise almost sure convergence  holds of course also for the operators $\pi_X(\gamma_t(0,n))f(x)$.
The union of the subspaces $\cD(\cH_\vare^\pi)$ for $ \vare > 0$ is norm-dense in $L^2_0(X)$ (assuming the action is ergodic). Hence, using the 
maximal inequality for $\gamma^\eta_t(0,n)$  
 in $L^2(X)$ stated in Theorem \ref{theorem:maximal}, we conclude that
$\pi_X(\gamma^\eta_t(0,n))f(x)\to 0$ almost everywhere for all $f\in L^2_0(X)$.
Clearly the same holds on the space of constants, 
unless $n=0$, in which case the limit is $\int_X fdm$. 

Now, since $L^\infty(X)\subset L^2(X)$ is $L^p$-norm-dense in every $L^p(X)$, 
we conclude that $L^p(X)$ has a norm-dense subspace on which pointwise almost sure convergence holds for $\pi_X(\gamma^\eta_t(0,n))f$. Using the strong type maximal inequality for $ 1 < p \le \infty$ for these operators, pointwise almost sure convergence holds for for every $f\in L^p(X)$. The same now follows for any p.m.p. action of $G$ using standard 
facts about ergodic decompositions.

 To complete the proof of Theorem \ref{thm:sl(2)}, it remains to establish pointwise convergence for the operators $\cM^\eta_{t}=\int_\RR \eta(t-s)\pi_X (m_K \ast a_{-s})ds$, acting in $L^p(X)$, $ 1 < p < \infty$. Note that the operators $\pi(\chi_n)$ are defined on $L^p$ and constitute projections  of norm at most one, satisfying $\pi(\chi_n)\pi(\chi_n)=\pi(\chi_n\ast \chi_n)=\pi(\chi_n)$, and also 
 $\pi(\chi_n)\pi(\chi_m)=\pi(\chi_n \ast \chi_m)=0$ if $n\neq m$. 
 By definition,   
if $f\in L^p(X)$ is $K$-finite then the linear span of its translates under $K$ is a finite-dimensional space $V(f)$, which is obviously $K$-invariant. The representation of $K$ in $V(f)$ is equivalent to a unitary representation, and $V(f)$ decomposes to a finite direct sum 
$\oplus_{n\in F_f}V_n(f)$, where the $K$-representation on $V_n(f)$ is via scalar multiplication by $\psi_n$, and $F_f$ is finite. Clearly, for every $n\in F_n$, each operator $\pi_X(\chi_n)$ maps $V(f)$ onto $V_n(f)$, and $\sum_{n\in F_f}\pi_X(\chi_n)$ acts as the identity on $V(f)$. By definition $f\in V(f)$,  and hence $f=\sum_{n\in F_f}\pi_X(\chi_n)f$. 
Therefore,  when $1 < p \le \infty$, for almost every $x\in X$ 
$$\cM_{t}^\eta f=\cM_{t}^\eta \left( \sum_{n\in F_f}\pi_X(\chi_n)f\right) 
=\sum_{n\in F_f}\int_\RR\eta(t-s) \pi_X(m_K \ast \delta_{a_{-s}})\pi_X(\chi_n)f(x)ds$$
$$
=\sum_{n\in F_f} 
\pi_X(\gamma_{-t}^{\eta^\vee}(0,n))f\longrightarrow \int_X f(y)d\mu(y)$$
where the final pointwise convergence result follows from Theorem \ref{theorem:pointwise} applied to the operators $\gamma_{-t}^{\eta^\vee}(0,n))$ with $\eta^\vee(s)=\eta(-s)$,  and $t\to \pm\infty$. As to the identification of the limit, if $0\notin F_f$, then the limit is $0=\int_X fd\mu$, and if $0\in F_f$, the limit is 
$\int_X \pi_X(\chi_0)fd\mu=\int_X fd\mu$, as stated. 


\subsection{ Pointwise convergence on a dense subspace}
We now turn to last remaining step, namely the construction of a dense subspace
of functions in $L^2(X)$ for which $\pi_X(\s_t(n,m))f(x)$ converges for almost all $x\in X$, using the arguments of \cite[\S 7.1]{N1}. First, for a smooth bump function $a(g)$ on $G$, let $\nu_a$ denote the absolutely continuous measure whose density w.r.t. Haar measure on $G$ is $a$. 
 Consider spectral subspaces $\cH_\vare$ defined above, and define the subspace $\cD(\cH_{\vare})=\{\pi_X(\nu_a)f=\int_G a(u)\pi_X(u)f du \space : \space a\in C_c^{\infty}(G), f\in \cH_{\vare}\}$.
 Note that for
 $f\in \cH_{\vare}$ we have the following norm estimate :
$$\norm{\pi_X(\s_t(0,n))f}_2^2=\int_{\tau\in \Sigma_{\pi}}\abs{\inn{f^\tau, v^\tau_n}}^2\abs{\Phi_{n,0}^\tau(a_t)}^2d\zeta_{\pi,f}(\tau)
$$
$$\le B^2(1+\abs{t})^2e^{-2\vare \abs{t}}\norm{f}_2^2$$
 For $h=\pi_X(\nu_a)f\in \cD(\cH_{\vare})$, we can consider 
 $$\frac{d}{dt}\pi(\s_t(0,n))h=\lim_{r\to 0}\frac{\pi_X(\s_{t+r}(0,n))h-\pi(\s_t(0,n))h}{r}$$
  where the limit is taken in the $L^2$-norm. The derivative is
 also contained in
 $\cH_{\vare}$, since $\cH_{\vare}$
 is a closed $\pi_X(G)$-invariant subspace. In fact the derivative at the point $t$ is equal to 
 $\int_G D^t_n(a)(u)\pi_X(u)f(x)du=\pi_X(\nu_{D^t_n(a)})f $, where $D^t_n(a)(u)$ is the smooth bump function $\frac{d}{dt}\left(\sigma_t(0,n)\ast \nu_a\right)(u)$.
 In particular, if $h$ belongs to $\cD(\cH_{\vare})$ its derivative $\frac{d}{dt}\pi(\s_t(0,n))h$
 is contained in $\cH_\vare$ also. 
 
 Furthermore, we claim that $t\mapsto \pi_X(\s_t(0,n))h(x)=\int_G\pi_X(\s_t(0,n)\ast \nu_a)f(x)du$ is continuous in $t$ for 
 $\mu$-almost all $x\in X$. 
 Indeed if $\alpha(u)$ is a compactly supported bounded function on $G$, then the following standard inequality holds (using e.g. Jensen's inequality) : 
 $$\int_X \abs{\int_G \alpha(u)f(u^{-1}x)du}^2 d\mu(x) \le \int_X \int_G \abs{\alpha(u)}\abs{f(u^{-1}x)}^2 dud\mu(x)$$
 $$=\norm{\alpha}_{L^1(G)}\norm{f}^2_{L^2(X)} < \infty $$
 Choosing $\alpha$ as the characteristic function of a large fixed ball $B$ in $G$, it follows that the function $u\mapsto  \abs{f(u^{-1}x)}^2$ restricted to $B$ is an $L^1$ function on $B$, for almost every $x\in X$. Hence if $\alpha_r$ are uniformly bounded continuous functions supported in $B$ such that $r\mapsto \alpha_r$ is continuous in the $L^1$-norm on $B$, 
 then for almost every $x$ (by Cauchy-Schwartz inequality) 
 $$\abs{\int_B (\alpha_t(u)-\alpha_s(u))f(u^{-1}x)du}$$
 $$\le \left(\int_B\abs{\alpha_t(u)-\alpha_s(u)}du\right)\left( \int_B\abs{\alpha_t(u)-\alpha_s(u)}\abs{f(u^{-1}x)}^2du\right)$$
 so that $t\mapsto \int_G \alpha_t(u)f(u^{-1}x)$ is continuous in $t$ for almost every $x$. This holds of course for $\alpha_t$ defined as the density of the measure $ \s_t(0,n))\ast a$. 
 

Fix $f\in \cD(\cH_{\vare})$, and let $w$
 denote an arbitrary vector in $L^2$. The function  $y_w(t)=\inn{\pi(\s_t(0,n))f,w}$ is differentiable (in fact $C^{\infty}$) in $t$, and since $f$ is a differentiable vector, 
 by the fundamental theorem of calculus :
$$\int_s^t\inn{\frac{d}{du}\pi(\s_u(0,n))f,w}du=\int_s^t\frac{d}{du}\inn{\pi(\s_u(0,n))f,w}du$$
$$
=\int_s^t \frac{d}{du}y_w(u)du=y_w(t)-y_w(s)=\inn{\int_s^t\frac{d}{du}\pi(\s_u(0,n))f du,w}\,.$$
 Since the equality is valid
 for all $w\in L^2$, we obtain by Fubini's theorem,
 for $\mu$-almost all $x\in X$ : 
$$\int_s^t \left(\frac{d}{du}\pi(\s_u(0,n))f\right)(x)du= \pi(\s_t(0,n))f(x)-\pi(\s_s(0,n))f(x)$$
Hence, if $t,s \ge N$ :
$$\abs{\pi(\s_t(0,n))f(x)-\pi(\s_s(0,n))f(x)}\le \int_N^{\infty}\abs{\frac{d}{du}\pi(\s_u(0,n))f(x)}du$$
Therefore, using the continuity of $t\mapsto \pi(\s_t(0,n))f(x)$ for almost all $x$, for any $N \ge 0$ we have 
$$\limsup\limits_{t,s \to \infty}\abs{\pi(\s_t(0,n))f(x)-\pi(\s_s(0,n))f(x)}\le \int_N^{\infty}\abs{\frac{d}{du}\pi(\s_u(0,n))f(x)}du$$ 
It follows that the set 
$$ \set{x : \limsup\limits_{t,s\to \infty}\abs{\pi(\s_t(0,n))f(x)-\pi(\s_s(0,n))f(x)} \ge \delta }$$ is contained in the set
$ \{x : \int_N^{\infty} \abs{\frac{d}{du}\pi(\s_u(0,n))f(x)}du \ge \delta \}. $
The measure of the latter set is estimated by
integrating over $X$. We obtain, for any $\delta > 0$ and $N\ge 0$, the bound : 
$$
  \frac{1}{\delta} \int_X\int_N^{\infty} \abs{\frac{d}{du}\pi(\s_u(0,n))f(x)}du d\mu (x)= $$
$$=\frac{1}{\delta} \int_N^{\infty}\inn{\abs{\frac{d}{du}\pi(\s_u(0,n))f(x)},1} du \le
 \frac{1}{\delta} \int_N^{\infty}\norm{\frac{d}{du}\pi(\s_u(0,n))f}_2 du \le$$
$$\le \frac{1}{\delta} \int_N^{\infty}\frac{1}{\vare}(1+u)e^{-\vare u}du\norm{f}_2$$
$$
 \le \frac{1}{\delta} \frac{1}{\vare^3}(1+N) e^{-\vare N}\norm{f}_2
 \mathop{\to}\limits_{N\to\infty}0. $$
It follows that the set where $\pi(\s_t(0,n))f(x)$ does not converge is
  a null set. Consequently,
 for $f\in \cD(\cH_{\vare})$, $\lim\limits_{t \to \infty}\pi(\s_t(0,n))f(x)$ exists for $\mu$-almost all $x\in X$ and is equal to $\int_X f(x)d\mu=0$, by the mean
ergodic theorem for these operators (which is an immediate consequence of convergence in norm).
 Finally, we just have to note that indeed $\bigcup_{\vare >
 0}\cD(\cH_{\vare})$ is norm-dense in $\{f\in L^2(X)\mid \int_Xfd\mu =0\}$, by its construction.
 Thus $\sigma_t(0,n)f$ converge pointwise for $f$ in a norm-dense subspace. 
 
 As already noted above it follows that $\gamma^\eta_t(0,n)f$ converges pointwise for $f$ in the same dense subspace, and the proof of the pointwise 
ergodic theorem for $\gamma_t(0,n)$ in $L^2(X)$ now follows from the maximal inequality.  \qed


\begin{thebibliography}{10}



%







\bibitem[Co2]{Co2} M. Cowling,
\textit{Sur les coefficients des repr\'esentations unitaires des
groupes de Lie simples}. Analyse harmonique sur les
groupes de Lie.  S\'eminaire Nancy--Strasbourg 1975--77, 
pp. 132--178,  Lecture Notes in Mathematics, \textbf{739}, 
Springer Verlag, 1979.



\bibitem[CHH]{CHH} M. Cowling, U. Haagerup and R. Howe
 Cowling, M.; Haagerup, U.; Howe, R. \textit{Almost L2 matrix coefficients}. J. Reine Angew. Math. 387 (1988), 97-110. 
 
 



%
\bibitem[EM1]{EM1} A.~Eskin and H.~Masur, \textit{Asymptotic formulas on flat surfaces}. Ergodic Theory Dynam. Systems 21 (2001), no. 2, 443-478. 



\bibitem[EMM]{EMM} A. Eskin, G. Margulis and S. Mozes, \textit{Upper bounds and asymptotics in a quantitative version of the Oppenheim conjecture.} Ann. of Math. (2) 147 (1998), no. 1, 93-141
%





\bibitem[GV]{GV} R. Gangolli and V. S. Varadarajan, 
\textit{Harmonic Analysis of Spherical Functions on Real Reductive
Groups.} Modern Surveys in Mathematics, \textbf{101}
Springer Verlag, 1988. 



%
%
%
%
%

\bibitem[H]{H} R. E. Howe,
\textit{On a notion of rank for unitary representations of the
classical groups.} Harmonic Analysis and Group
Representations, C.I.M.E., 2$^{\,\circ}$ ciclo, Liguori (1982) pp. 223--232.  


\bibitem[HT]{HT} R. E. Howe, and E. C. Tan, 
\textit{Non-Abelian Harmonic Analysis.} Springer
Verlag, 1992. 



%
%

\bibitem[Kn]{Kn} A. W. Knapp,
\textit{Representation Theory of Semisimple Groups: an Overview
Based on Examples.} Princeton Mathematical Series, \textbf{36}
 Princeton Univ. Press, 1986.




\bibitem[N1]{N1}  A. Nevo, {\it  Pointwise Ergodic Theorems for Radial Averages on Simple Lie Groups I.}  Duke
 Journal of Mathematics, vol. 76, no. 1, pp. 113-140 (1994).


\bibitem[N2]{N2} A. Nevo, {\it Pointwise Ergodic Theorems for Radial Averages on Simple Lie Groups II.} Duke Journal of Mathematics, 
vol. 86, no. 2, pp. 239-259 (1997).




\bibitem[NS]{NS}  A. Nevo and E. M. Stein, 
\textit{ Analogs of Wiener's  Ergodic Theorems for 
Semi-Simple Groups I.} 
Annals of Mathematics, vol. 145, pp. 565-595 (1997).

%
%
%
%
%
%
%
%
%
%
%
%
%

%
%

\bibitem[Ve1]{Ve1} W. A. Veech, \textit{Siegel measures}. Ann. of Math. (2) 148 (1998), no. 3, 895-944

\end{thebibliography}
\end{document}